\documentclass[11pt,a4paper]{article}

\usepackage[T1]{fontenc}
\usepackage[utf8]{inputenc}
\usepackage{amsmath,amsthm,amssymb}
\usepackage{enumitem}
\usepackage[colorlinks=true,linkcolor=blue,citecolor=blue,urlcolor=blue]{hyperref}

\newcommand{\Nil}{\operatorname{Nil}}

\newcommand{\irad}[1]{\sqrt[\mathrm{i}]{#1}}
\newcommand{\grad}[1]{\sqrt[\mathrm{g}]{#1}}
\newcommand{\jrad}[1]{\sqrt[\mathrm{j}]{#1}}

\theoremstyle{plain}
\newtheorem{theorem}{Theorem}[section]
\newtheorem{corollary}[theorem]{Corollary}

\newtheorem{proposition}[theorem]{Proposition}

\newtheorem*{conjecture*}{Conjecture}

\newtheorem{definition}{Definition}[section]

\newtheorem{example}{Example}[section]

\numberwithin{equation}{section}

\title{Strong and Explicit Forms of the One-Sided Nullstellensatz over Division Rings} 
\author{Masood Aryapoor\\
	\tiny{\textit{Department of Business and Mathematics}}\\
	\tiny{\textit{M\"{a}lardalen  University}}\\
	\tiny{\textit{Hamngatan 15, 632 17, Eskilstuna, 
			Sweden
	}}
}
\date{}
\begin{document}
	\maketitle
	\begin{abstract} 
		We establish strong and explicit forms of a one-sided Nullstellensatz for polynomial rings \(D[x_1,\dots,x_n]\) over arbitrary division rings \(D\) with central indeterminates. To this end, we employ the theory of integral dependence over left ideals and the notion of G-left ideals, which plays a role analogous to that of G-ideals in the classical theory. We prove that the Jacobson radical of any left ideal of \(D[x_1,\dots,x_n]\) coincides with the intersection of all G-left ideals containing it. As a consequence, we obtain an explicit description of Jacobson radicals of left ideals of polynomial rings over division rings in terms of integral dependence.
	\end{abstract}
	\section{Introduction}
	From an algebraic perspective, the classical Nullstellensatz  admits three standard forms over a field \(k\). The weak form characterizes the maximal ideals of the polynomial ring \(k[x_1,\dots,x_n]\):
	\begin{theorem}[Weak Form of Classical Nullstellensatz] 
		For every maximal ideal \(M\) of \(k[x_1,\dots,x_n]\), the quotient ring \(k[x_1,\dots,x_n]/M\) is a finite field extension of \(k\). 
	\end{theorem}
	The strong form asserts that, for every ideal \(I\) of \(k[x_1,\dots,x_n]\), the Jacobson radical and the prime radical of \(I\) coincide: 
	\begin{theorem}[Strong Form of Classical Nullstellensatz] 
		For every ideal \(I\) of \(k[x_1,\dots,x_n]\),
		\[
		\bigcap_{\substack{I\subseteq M}}M = \bigcap_{\substack{I\subseteq P}}P,
		\]
		where the first intersection is taken over maximal ideals \(M\subset k[x_1,\dots,x_n]\) and  the second intersection over prime ideals \(P\subset k[x_1,\dots,x_n]\).
	\end{theorem}
	The third form is the classical Hilbert Nullstellensatz:
	\begin{theorem}[Explicit Form of Classical Nullstellensatz] 
		For every ideal \(I\) of \(k[x_1,\dots,x_n]\),
		\[
		\bigcap_{\substack{I\subseteq M}}M = \{f\in k[x_1,\dots,x_n]\mid (\exists N\ge 1)\quad f^N\in I\},
		\]
		where the intersection is taken over maximal ideals \(M\subset k[x_1,\dots,x_n]\).
	\end{theorem}
	The three forms above suggest natural analogues in the noncommutative setting, where ideals are replaced by one-sided ideals and polynomial rings over fields by polynomial rings over division rings. The main goal of this paper is to investigate the one-sided versions of the above results for polynomial rings over division rings in central variables, which we shall call \emph{one-sided Nullstellensatz}. Let \(D\) be a division ring and consider the polynomial ring \(D[x_1,\dots,x_n]\), where the indeterminates are central. The Amitsur--Small theorem \cite{AmitsurSmall1978}  gives the weak form of the one-sided Nullstellensatz:
	\begin{theorem}[Weak Form of One-Sided Nullstellensatz] 
		For every maximal left ideal \(M\) of \(D[x_1,\dots,x_n]\), \(D[x_1,\dots,x_n]/M\) is finite dimensional as  a left vector space over \(D\). 
	\end{theorem}
	In \cite{AlonNullestellensatz2021}, Alan and Paran established the following strong form of the one-sided Nullstellensatz for the division ring of quaternions \(\mathbb{H}\). 
	\begin{theorem}\label{thm:strong-null-H-comp-prime}
		For every left ideal \(I\) of \(\mathbb{H}[x_1,\dots,x_n]\),
		\[
		\bigcap_{\substack{I\subseteq M}}M = \bigcap_{\substack{I\subseteq P}}P,
		\]
		where the first intersection is taken over maximal left ideals \(M\subset \mathbb{H}[x_1,\dots,x_n]\) and  the second intersection over completely prime left ideals \(P\subset \mathbb{H}[x_1,\dots,x_n]\).
	\end{theorem}
	Recall that a proper left ideal \(P\) is completely prime if 
	\[
	ab\in P \text{ and } Pb\subseteq P\implies a\in P\text{ or }b\in P.\]
	This notion was introduced and studied by Reyes in \cite{reyes2010one}. 
	
	An explicit form of the one-sided Nullstellensatz for \(\mathbb H\) was obtained in \cite{aryapoor2024explicit} (see also \cite{aryapoor2025central} for a more general result).
	\begin{theorem}[Explicit Form of One-Sided Nullstellensatz for \(\mathbb{H}\)]
		For every left ideal 
		\(I\subseteq \mathbb{H}[x_1,\ldots,x_n]\), 
		\[
		\bigcap_{\substack{I\subseteq M}}M = \{f\in \mathbb{H}[x_1,\ldots,x_n]\mid (\forall q\in \mathbb{H})(\exists N\geq1)\quad (qf)^N\in\sum_{i=0}^{N}I(qf)^i\},
		\]
		where the intersection is taken over maximal left ideals \(M\subset \mathbb{H}[x_1,\dots,x_n]\). 
	\end{theorem}
	
	In \cite{cimprivc2025parallels}, Cimpri\v{c} obtained another version of the strong form for the quaternions, which was later generalized to finite-dimensional algebras by Cimpri\v{c} and Sch\"otz in \cite{Cimpric2026}. Specializing their result to the case of division rings yields the following theorem.  
	\begin{theorem} \label{thm:strong-null-D-prime-semprime}
		Let \(D\) be a centrally finite division ring. Then for every left ideal \(I\) of \(D[x_1,\dots,x_n]\),
		\[
		\bigcap_{\substack{I\subseteq M}}M = \bigcap_{\substack{I\subseteq P}}P =  \bigcap_{\substack{I\subseteq Q}}Q,
		\]
		where the first intersection is taken over maximal left ideals \(M\subset D[x_1,\dots,x_n]\),  the second intersection over prime left ideals \(P\subset D[x_1,\dots,x_n]\), and the third intersection over semiprime left ideals \(Q\subset D[x_1,\dots,x_n]\).
	\end{theorem}
	Recall that a proper left ideal \(P\) of a ring \(R\) is prime (resp., semiprime) if 
	\[
	aRb\subseteq P \implies a\in P \text{ or } b\in P\quad  (\text{resp., } aRa\subseteq P \implies a\in P).
	\] 
	Neither Alan and Paran's strong form (Theorem~\ref{thm:strong-null-H-comp-prime}) nor Cimpri\v{c} and Sch\"otz's strong form (Theorem~\ref{thm:strong-null-D-prime-semprime}) holds for an arbitrary division ring, even when \(n=1\) (see Example~\ref{exm:counter-com-prim}). 
	
	In this paper, we present a strong form of the one-sided Nullstellensatz over an arbitrary division ring and derive a general explicit form. Our approach is based on the notion of integral dependence over one-sided ideals developed in \cite{aryapoor2026integral,aryapoor2026integral2}.  Recall that an element \(a\) of a ring \(R\) is integral over a left ideal \(I\) of \(R\) if there exist \(r_0,\dots,r_{n-1}\in I\) such that 
	\[ a^n+r_{n-1}a^{n-1}+\cdots+r_1a+r_0=0. \]
	A left ideal \(P\) of \(R\) is called a \emph{G-left ideal} if there exists an element \(a\in R\) that is not integral over \(P\) but is integral over every left ideal properly containing \(P\).  Our main results are the following theorems, whose proofs are given in Subsection~\ref{subsec:proofs}. 
	\begin{theorem}[Strong Form of One-Sided Nullstellensatz] \label{thm:stong-null-D}
		Let \(D\) be a division ring. For every left ideal \(I\) of \(D[x_1,\dots,x_n]\),
		\[
		\bigcap_{\substack{I\subseteq M}}M = \bigcap_{\substack{I\subseteq P}}P,
		\]
		where the first intersection is taken over maximal left ideals \(M\subset D[x_1,\dots,x_n]\) and  the second intersection over G-left ideals \(P\subset D[x_1,\dots,x_n]\).
	\end{theorem}
	\begin{theorem}[Explicit Form of One-Sided Nullstellensatz] \label{thm:explicit-null-D}
		Let \(D\) be a division ring. For every left ideal \(I\) of \(R=D[x_1,\dots,x_n]\),
		\[
		\bigcap_{\substack{I\subseteq M}}M = \{f\in R\mid (\forall g\in R)(\exists N\geq1)\quad (gf)^N\in\sum_{i=0}^{N-1}I(gf)^i\},
		\]
		where the intersection is taken over maximal left ideals \(M\subset D[x_1,\dots,x_n]\).
	\end{theorem}
	A classical approach to the Nullstellensatz is based on the notion of a Jacobson ring (see, for example, \cite{kaplansky2006commutative}). Following this approach, we introduce the notion of a G-left Jacobson ring and use it to formulate our strong and explicit forms of the one-sided Nullstellensatz. This enables us to extend the one-sided Nullstellensatz to a more broader class of rings (see Subsection~\ref{subsec:fully}). 
	
	Section~\ref{sec:back-pre} provides the necessary preliminaries, and the main results are presented in Section~\ref{sec:one-sided}. Throughout the paper, all rings are assumed to be associative and unital.

	\section{Preliminaries}\label{sec:back-pre}
	This section presents the necessary preliminaries, most of which are taken from \cite{aryapoor2026integral, aryapoor2026integral2}. For proofs of the stated results, we refer the reader to \cite{aryapoor2026integral, aryapoor2026integral2}. Throughout this section, let \(R\) denote a ring.

	An element \(a\) of  \(R\) is \emph{integral} over a left ideal \(I\subseteq R\) if there exists \(n\geq 1\) such that \(a^n\in \sum_{i=0}^{n-1}Ia^i\), that is,  
	\[ a^n+r_{n-1}a^{n-1}+\cdots+r_1a+r_0=0, \]
	for some \(r_0,\dots,r_{n-1}\in I\). 
	A subset is \emph{integral} over \(I\) if each of its elements is. 
	A left ideal \(I\) of a ring \(R\) is said to be \emph{integrally closed} if no left ideal of \(R\) properly containing \(I\) is integral over \(I\). The \emph{integral radical} of a left ideal \(I\), denoted by \(\irad{I}\), is the intersection of all integrally closed left ideals of \(R\) that contain \(I\). The \emph{Jacobson radical} of a left ideal \(I\), denoted by \(\jrad{I}\), is defined as the intersection of all maximal left ideals containing \(I\). 
	\begin{proposition}\label{prp:int-in-intrad}
		Let \(I\) and \(J\) be left ideals of \(R\). If \(J\) is integral over
		\(I\), then \(J \subseteq \irad{I}.\)
	\end{proposition} 
	
	\begin{theorem}\label{thm:rad-artinian-algebra}
		Let \(R\) be a ring satisfying one of the following properties:
		\begin{enumerate}
			\item \(R\) is left artinian.
			\item \(R\) is an algebraic algebra over a field \(k\).
			\item \(R\) is an algebra over a field \(k\) with \(\dim_k R < |k|\).
		\end{enumerate} 
		Then, for any left ideal \(I\) in \(R\), the left ideal \(\jrad{I}\) is integral over \(I\).  
	\end{theorem}
	
	A left ideal \(P\) of \(R\) is called  a \emph{G-left ideal with witness} \(a\in R\) if \(a\) is not integral over \(P\) and is integral over every left ideal \(J\supsetneq P\). A left ideal \(P\) is called a \emph{G-left ideal} if there exists an element \(a\in R\) such that \(P\) is a G-left ideal with witness \(a\). 
	\begin{proposition}\label{prp:grad-quotient-ring}
		Let \(J\) be a two-sided ideal of \(R\). Then, for every left ideal
		\(I\) of \(R\) containing \(J\),
		\[ \grad{\,I/J\,}=\frac{\grad{I}}{J}.
		\]
	\end{proposition}
	\begin{proposition}\label{prp:1-1corres-G-left}
		Let \(a\in R\). Then the assignment \(P\mapsto P[t]+ R[t](1-at)\) establishes a 1-1 correspondence between the set of all G-left ideals \(P\subset R\) with witness \(a\)   and the set of all maximal left ideals of \(R[t]\) containing \(1-at\). Furthermore, for any G-left ideal  \(P\subset R\) with witness \(a\),
		\[
		\left(P[t]+ R[t](1-at)\right)\cap  R = P.  
		\]
	\end{proposition}
	The \emph{G-radical} of a left ideal \(I\) of \(R\), denoted by \(\grad{I}\), is the intersection of all G-left ideals of \(R\) containing \(I\). 
	\begin{proposition}\label{prp:incl-grad-irad-jrad}
		For any left ideal \(I\) of \(R\), we have
		\[
		\grad{I}\subseteq \irad{I}\subseteq \jrad{I}.
		\]
	\end{proposition}
	\begin{proposition}\label{prp:grad-semi-int}
		For every left ideal \(I\) of \(R\), the left ideal \(\grad{I}\) is integral over \(I\).
		
	\end{proposition}
	
	\begin{proposition}\label{prp:gen-kothe-1st-equiv}
		For a ring \(R\), the following statements are equivalent.
		\begin{enumerate}[label=(\roman*)]
			\item For every left ideal \(I\) of \(R\), \(\irad{I}\) is integral over \(I\).
			\item \(\grad{I}=\irad{I}\) for every left ideal \(I\) of \(R\).
		\end{enumerate}
	\end{proposition}
	Recall that the upper nilradical \(\Nil^*(R)\) of \(R\) is the largest nil two-sided ideal of \(R\). 
	\begin{theorem}\label{thm:G-radical.2-sided}
		For any ring \(R\), the intersection of all the G-left ideals of \(R\) is equal to the upper nilradical \(\Nil^*(R)\).
	\end{theorem}

	\section{The one-sided Nullstellensatz}\label{sec:one-sided}
	This part provides the main results of the paper, namely the strong and explicit forms of the one-sided Nullstellensatz over arbitrary division rings.  The final subsection establishes a generalization of the one-sided Nullstellensatz for a broader class of rings than division rings. 
	
	\subsection{G-left Jacobson rings}
	In the commutative setting, a ring \(R\) is called a \emph{Jacobson ring} if every prime ideal of \(R\) is an intersection of maximal ideals; equivalently, for every prime ideal \(P\), \(\jrad{P}=P\). This definition has a direct generalization to the noncommutative setting: a ring is called \emph{Jacobson} if \(\jrad{P}=P\) for every prime two-sided ideal \(P\).  
	
	It is known that a commutative ring \(R\) is a Jacobson ring iff the Jacobson  radical of any ideal of \(R\) coincides with the G-radical of that ideal \cite[Theorem 25 and Theorem 30]{kaplansky2006commutative}. Generalizing this property to the noncommutative setting, we introduce the following
	notion.
	\begin{definition}
		A ring \(R\) is called \emph{G-left Jacobson} if  \(\grad{I}=\jrad{I}\) for every left ideal \(I\subseteq R\).
	\end{definition}
	For commutative rings, the notion of a G-left Jacobson ring coincides with that of a Jacobson ring. We record some elementary facts about G-left Jacobson rings. 
	\begin{proposition}
		Let \(R\) be  a ring. Then
		\begin{enumerate}
			\item \(R\) is G-left Jacobson iff  \(\jrad{P} = P\) for every G-left ideal \(P\).
			\item Let \(J\) be a two-sided ideal of \(R\). If \(R\) is G-left Jacobson, then so is \(R/J\).
		\end{enumerate}
	\end{proposition}
	\begin{proof}
		The proof of the first statement is easy, and the second statement follows from Proposition~\ref{prp:grad-quotient-ring}. 
	\end{proof}
	\begin{proposition}\label{prp:irad-G-left}
		Let \(R\) be a G-left Jacobson ring. Then, for any left ideal \(I\subset R\), we have
		\[ \grad{I}=\irad{I}=\jrad{I}.\]
	\end{proposition}
	\begin{proof}
		The equalities follow from Proposition~\ref{prp:incl-grad-irad-jrad}. 
		
	\end{proof}
	
	\begin{proposition}\label{prp:equiv-G-left}
		A ring \(R\) is a G-left Jacobson ring iff  \(\jrad{I}\) is integral over \(I\) for all left ideals \(I\) of \(R\).  
	\end{proposition}
	\begin{proof}
		Suppose \(R\) is G-left Jacobson. By Proposition~\ref{prp:grad-semi-int}, \(\grad{I}\) is integral over \(I\) for any left ideal \(I\subseteq R\). Since  \(\jrad{I}=\grad{I}\), the result follows. 
		
		Conversely, let \(I\) be a left ideal of \(R\). Since \(\jrad{I}\) is integral over \(I\), we have \(\jrad{I}\subseteq \irad{I}\) by Proposition~\ref{prp:int-in-intrad}. The reverse inclusion holds for all rings by Proposition~\ref{prp:incl-grad-irad-jrad}. Therefore, \(\jrad{I} = \irad{I}\). In particular, \(\irad{I}\) is integral over \(I\). Proposition~\ref{prp:gen-kothe-1st-equiv} gives \(\grad{I} = \irad{I}\). We conclude that \(\grad{I} = \irad{I} = \jrad{I}\), which completes the proof.  
	\end{proof}
	As an application, we give the following result. 
	\begin{corollary}\label{cor:ex-G-left}
		Let \(R\) be a ring satisfying one of the following properties:
		\begin{enumerate}
			\item \(R\) is left artinian.
			\item \(R\) is an algebraic algebra over a field \(k\).
			\item \(R\) is an algebra over a field \(k\) with \(\dim_k R < |k|\).
		\end{enumerate} 
		Then \(R\) is G-left Jacobson.  
	\end{corollary}
	\begin{proof}
		By Theorem~\ref{thm:rad-artinian-algebra},  \(\jrad{I}\) is integral over \(I\) for every left ideal \(I\) of \(R\). The result now follows from  Proposition~\ref{prp:equiv-G-left}.
	\end{proof}
	
	A result of Watters states that if a ring \(R\) is a Jacobson ring (in the two-sided sense), then \(R[x]\) is also Jacobson \cite{watters1975polynomial}.  The following example demonstrates that this property does not hold for G-left Jacobson rings.
	\begin{example}
		In \cite{SmoktunowiczPuczylowski2001}, Smoktunowicz and Puczy{\l}owski constructed a nil algebra \(A\)
		over a countable field \(k\) such that \(A[x]\) is Jacobson radical but
		is not nil. Let
		\(
		R=k1\oplus A
		\)
		be its unitization. Since every element outside \(A\) is invertible,
		\(A\) is the unique maximal left ideal of \(R\). Moreover,
		\(\operatorname{Nil}^{*}(R)=A\). It follows from Theorem~\ref{thm:G-radical.2-sided} and Proposition~\ref{prp:incl-grad-irad-jrad} that for every proper left
		ideal \(I\subset R\),
		\[
		A=\operatorname{Nil}^{*}(R)
		=\grad{(0)}
		\subseteq\grad{I}
		\subseteq\jrad{I}
		=A.
		\]
		Thus \(R\) is G-left Jacobson.
		
		On the other hand, \(A[x]\) is a Jacobson radical two-sided
		ideal of \(R[x]\), and
		\[
		R[x]/A[x]\cong k[x]
		\]
		has zero Jacobson radical. Consequently, \(\jrad{(0)}\) coincides with \(A[x]\),
		which is not nil. Since the G-radical of the zero ideal is the upper
		nilradical by Theorem~\ref{thm:G-radical.2-sided}, the radicals computed in \(R[x]\) satisfy
		\[
		\grad{(0)}
		=\operatorname{Nil}^{*}(R[x])
		\subsetneq A[x] = \jrad{(0)}.
		\]
		Therefore, \(R[x]\) is not a G-left Jacobson ring, while \(R\) is. 
	\end{example}
	
	\subsection{The one-sided Nullstellensatz over division rings}\label{subsec:proofs}
	In what follows, let \(D\) be an arbitrary division ring, and let \(D[x_1,\dots,x_n]\) denote the polynomial ring over \(D\) in \(n\) central indeterminates. 
	
	\begin{theorem}\label{thm:divsion-G-left}
		Every \(G\)-left ideal of \(D[x_1,\dots,x_n]\) is a finite intersection
		of maximal left ideals. In particular, \(D[x_1,\dots,x_n]\) is
		G-left Jacobson.
	\end{theorem}
	
	\begin{proof}
		Let \(P\) be a \(G\)-left ideal of \(R=D[x_1,\dots,x_n]\) with witness \(a\).
		By Proposition~\ref{prp:1-1corres-G-left}, the left ideal 
		\[
		M=P[t]+R[t](1-at)
		\]
		is a maximal left ideal of \(R[t]\), where \(t\) is central,
		and \(M\cap R=P\). Note that 
		\(R[t]=D[x_1,\dots,x_n,t]\). 
		By \cite[Proposition 5.3]{robson1981liberal}, 
		\[
		M\cap D[x_1,\dots,x_n] = M\cap R = P
		\]
		is a finite intersection of maximal left ideals of \(R\). In particular, \(\jrad{P}=P\). This completes the proof. 
	\end{proof}
	We can now prove Theorem~\ref{thm:stong-null-D} and Theorem~\ref{thm:explicit-null-D}.
	\begin{proof}[Proof of  Theorem~\ref{thm:stong-null-D}] 
		By Theorem~\ref{thm:divsion-G-left}, \(D[x_1,\dots,x_n]\) is G-left Jacobson, and consequently, \(\jrad{I} = \grad{I}\) for every left ideal \(I\) of \(D[x_1,\dots,x_n]\). This gives the desired equality in Theorem~\ref{thm:stong-null-D}.  
	\end{proof}
	\begin{proof}[Proof of  Theorem~\ref{thm:explicit-null-D}] 
		Let \(I\subseteq D[x_1,\dots,x_n]\) be a left ideal. By Proposition~\ref{prp:equiv-G-left}, \(\jrad{I}\) is  integral over \(I\). This proves the inclusion
		\[
		\bigcap_{\substack{I\subseteq M\\M: \text{maximal left ideal}}}M  \subseteq \{f\in R\mid (\forall g\in R)(\exists N\geq1)\quad (gf)^N\in\sum_{i=0}^{N-1}I(gf)^i\}.
		\]
		To prove the reverse inclusion, we note that if \(f\in D[x_1,\dots,x_n]\) belongs to the right-hand side set, then the left ideal \(D[x_1,\dots,x_n]f\) is integral over \(I\).  Proposition~\ref{prp:incl-grad-irad-jrad} gives
		\(D[x_1,\dots,x_n]f\subseteq \irad{I}\). By Theorem~\ref{thm:divsion-G-left}, \(D[x_1,\dots,x_n]\) is G-left Jacobson. It follows from  Proposition~\ref{prp:irad-G-left} that \(\jrad{I} = \irad{I}\). Therefore \(f\in \jrad{I}\), which completes the proof. 
	\end{proof}
	In \cite{chapman2025amitsur}, Chapman and Paran called a division ring \emph{Amitsur--Small} if for all \(n\ge m\geq 1\), the  contraction to \(D[x_1,\dots,x_{m}]\) of every maximal left ideal of \(D[x_1,\dots,x_n]\) is a maximal left ideal.  For such division rings, we have the following stronger result. 
	\begin{proposition}
		Let \(D\) be an Amitsur--Small division ring. Then every G-left ideal of \(D[x_1,\dots,x_n]\) is a maximal left ideal. 
	\end{proposition}
	
	\begin{proof} 
		We just note that in the proof of Theorem~\ref{thm:divsion-G-left}, the G-left ideal \(P\) is a contraction of a maximal left ideal of \(D[x_1,\dots,x_n,t]\), from which the result follows.  
	\end{proof}
	
	The following example shows that Theorem~\ref{thm:strong-null-H-comp-prime}  and Theorem~\ref{thm:strong-null-D-prime-semprime} do not hold for arbitrary division rings. 
	\begin{example}\label{exm:counter-com-prim}
		In \cite[Section 10]{lam2004wedderburn}, Lam and Leroy constructed a division ring \(K\) with a monic quadratic polynomial \(f(x)\in R = K[x]\) satisfying the following properties: 
		\begin{enumerate}
			\item \(f(x)\) has a unique factorization 
			\[
			f(x) = (x-c)(x-b)
			\]
			into monic irreducible polynomials in \(K[x]\).
			\item The elements \(c,b\in K\) are nonconjugate, or equivalently, the left \(R\)-modules \(R/R(x-c)\) and \(R/R(x-b)\) are not isomorphic.  
		\end{enumerate}
		Note that \(K[x]\) is a principal (left and right) ideal domain. 
		By \cite[Proposition 3.3]{aryapoor2025prime}, the above properties imply that the left ideal \(I = Rf\) is completely prime.  Proposition 3.7 in \cite{aryapoor2025prime} states that any completely prime left ideal of a principal ideal domain is prime. So \(I\) is prime. The first property above implies that there exists exactly one maximal left ideal containing \(I\), namely \(R(x-b)\).  Therefore, \(\jrad{I} = R(x-b)\). As \(I\) is both prime (hence semiprime) and completely prime,  we conclude that Theorem~\ref{thm:strong-null-H-comp-prime}  and Theorem~\ref{thm:strong-null-D-prime-semprime} are not valid for the division ring \(K\) and \(n=1\).

	\end{example}
	\subsection{The one-sided Nullstellensatz for fully bounded noetherian Jacobson rings}\label{subsec:fully}
	The work of Resco, Stafford and Warfield \cite{resco1986fully} allows us to extend the
	one-sided Nullstellensatz from division rings to a broader class
	of coefficient rings. We need some preliminaries. For more details, see \cite[Section 9]{goodearl2004introduction}. 
	
	Recall that a left ideal of a ring is essential if it has nonzero
	intersection with every nonzero left ideal. A ring is left bounded
	if every essential left ideal contains a nonzero two-sided ideal.
	A ring is left fully bounded if each of its prime factor rings is
	left bounded. The right-handed notions are defined analogously.
	A ring is fully bounded if it is both left and right fully bounded. A ring is noetherian if it is both left and right noetherian. 
	\begin{theorem}\label{thm:fully-bounded-G-left}
		Let \(R\) be a fully bounded, noetherian Jacobson
		ring. Then, every \(G\)-left ideal
		of \(R[x_1,\dots,x_n]\), with central indeterminates, is a finite
		intersection of maximal left ideals. In particular,
		\(R[x_1,\dots,x_n]\) is \(G\)-left Jacobson.
	\end{theorem}
	
	\begin{proof}
		Set \(B=R[x_1,\dots,x_n]\), and let \(P\) be a \(G\)-left ideal
		of \(B\) with witness \(a\in B\). By Proposition~\ref{prp:1-1corres-G-left},
		\[
		M=P[t]+B[t](1-at)
		\]
		is a maximal left ideal of \(B[t]\), where \(t\) is central, and
		\(
		M\cap B=P.
		\)
		Thus \(V=B[t]/M\) is a simple left \(B[t]\)-module. Since \(B[t]=R[x_1,\dots,x_n,t]\), it follows from \cite[Corollary 17.11]{goodearl2004introduction} that \(V\) is
		a finitely generated semisimple
		left \(R\)-module. In particular, \(V\) is of finite length as a left \(B\)-module. It follows that \(V\) contains a simple \(B\)-submodule \(W\). Then 
		\[
		\sum_{j\ge 0} t^jW
		\]
		is a nonzero \(B[t]\)-submodule of \(V\). Since \(V\) is simple as a \(B[t]\)-module, 
		\[
		V = \sum_{j\ge 0} t^jW.
		\]
		Therefore, \(V\) is semisimple as a \(B\)-module. Since \(V\) is also  of finite length as a left \(B\)-module, we can write
		\[
		V = V_1\oplus\cdots\oplus V_m,
		\]
		where each \(V_i\) is a simple \(B\)-submodule of \(V\). Writing \(1+M\in V\) as a sum \(\sum_j v_j\), where \(v_j\in V_j\), we see that 
		\[
		M = \operatorname{ann}_{B[t]}(1+M) = \bigcap_{j=1}^m  \operatorname{ann}_{B[t]}(v_j). 
		\]
		Since \(M\cap R = P\) and each \(\operatorname{ann}_{B[t]}(v_j)\cap R\) is either \(B\) or a maximal left ideal of \(B\), the result follows. 
	\end{proof}
	
	As an example, any polynomial ring \(R[x_1,\dots,x_n]\) over an artinian ring \(R\) is G-left Jacobson. 
	
	\bibliographystyle{plain}
	\bibliography{references}

\end{document}